\documentclass[11pt]{amsart}

\theoremstyle{plain}
\newtheorem{theorem}                {Theorem}      [section]
\newtheorem{proposition}  [theorem]  {Proposition}

\theoremstyle{definition}

\newtheorem{remark}       [theorem]  {Remark}
\newtheorem{definition}   [theorem]  {Definition}

\DeclareMathOperator{\trace}{trace}

\DeclareMathOperator{\Span}{span}
 
\DeclareMathOperator{\cst}{constant}

\numberwithin{equation}{section}

\begin{document}

\title[Biharmonic curves in Sasakian space forms]{A note on biharmonic curves in Sasakian space forms}

\author[D.~Fetcu]{Dorel Fetcu}

\address{Department of Mathematics\\
"Gh. Asachi" Technical University of Iasi\\
Bd. Carol I no. 11 \\
700506 Iasi, Romania} \email{dfetcu@math.tuiasi.ro}

\begin{abstract} We classify the biharmonic non-Legendre curves in a
Sasakian space form for which the angle between the tangent vector
field and the characteristic vector field is constant and obtain
explicit examples of such curves in $\mathbb{R}^{2n+1}(-3)$.
\end{abstract}

\date{}



\thanks{The author wants to thank to Prof. C. Oniciuc for
many discussions and useful suggestions.}

\thanks{The author was partially supported by the Grant CEEX, ET, 5871/2006, Romania.}

\maketitle

\section{Introduction}
\setcounter{equation}{0}

In 1964, J.~Eells and J.H.~Sampson introduced the notion of
poly-harmonic maps as a natural generalization of harmonic maps
(\cite{Eells}). Thus,  while \textit{harmonic maps} between
Riemannian manifolds $\phi:(M,g)\to(N,h)$ are the critical points
of the \textit{energy functional}
$E(\phi)=\frac{1}{2}\int_{M}|d\phi|^{2} \ v_{g}$, the
\textit{biharmonic maps} are the critical points of the
\textit{bienergy functional}
$E_{2}(\phi)=\frac{1}{2}\int_{M}|\tau(\phi)|^{2} \ v_{g}$.

On the other hand, B.-Y. Chen defined the biharmonic submanifolds
in an Euclidean space as those with harmonic mean curvature vector
field (\cite{Chen}). If we apply the characterization formula of
biharmonic maps to Riemannian immersions into Euclidean spaces, we
recover Chen's notion of biharmonic submanifold.

The Euler-Lagrange equation for the energy functional is
$\tau(\phi)=0$, where $\tau(\phi)=\trace\nabla d\phi$ is the
tension field, and the Euler-Lagrange equation for the bienergy
functional was derived by G. Y. Jiang in \cite{Jiang}:
$$
\begin{array}{cl}
\tau_{2}(\phi)&=-\Delta\tau(\phi)-\trace\
R^{N}(d\phi,\tau(\phi))d\phi\\ \\
&=0.\end{array}
$$
Since any harmonic map is biharmonic, we are interested in
non-harmonic biharmonic maps, which are called
\textit{proper-biharmonic}.

There are several classification results and some methods to
construct biharmonic submanifolds in space forms
(\cite{MontaldoOniciuc}, \cite{BMO}). In a natural way, the next
step is the study of biharmonic submanifolds in Sasakian space
forms. Thus, J. Inoguchi classified in \cite{Ino} the
proper-biharmonic Legendre curves and Hopf cylinders in a
3-dimensional Sasakian space form $M^{3}(c)$, and in
\cite{FetcuOniciuc} the explicit parametric equations were
obtained. In \cite{CIL}, J.T. Cho, J. Inoguchi and J.-E. Lee
studied the biharmonic curves in a 3-dimensional Sasakian space
forms and T. Sasahara studied the biharmonic integral surfaces in
5-dimensional Sasakian space forms (\cite{Sasahara1}). New
classification results for biharmonic Legendre curves and examples
of proper-biharmonic submanifolds in any dimensional Sasakian
space form were obtained in \cite{FetcuOniciuc2}.

Recent results on biharmonic submanifolds in spaces of nonconstant
sectional curvature were obtained by T. Ichiyama, J. Inoguchi and
H. Urakawa in \cite{Ura}, by Y.-L. Ou and Z.-P. Wang in \cite{Ou},
and by W. Zhang in \cite{Zhang}.

Biharmonic submanifolds in pseudo-Euclidean spaces were also
studied, and many examples and classification results were
obtained (for example, see \cite{Arv}, \cite{Chen}).

The goals of our paper are to obtain new classification results
for biharmonic non-Legendre curves in any dimensional Sasakian
space form and to obtain explicit equations for some of such
curves in $\mathbb{R}^{2n+1}(-3)$.

For a general account of biharmonic maps see
\cite{MontaldoOniciuc} and \textit{The Bibliography of Biharmonic
Maps} \cite{bibl}.

\noindent \textbf{Conventions.} We work in the $C^{\infty}$
category, that means manifolds, metrics, connections and maps are
smooth. The Lie algebra of the vector fields on $M$ is denoted by
$C(TM)$.

\section{Preliminaries}
\setcounter{equation}{0}

A triple $(\varphi,\xi,\eta)$ is called a \textit{contact
structure} on a manifold $N^{2n+1}$, where $\varphi$ is a tensor
field of type $(1,1)$ on $N$, $\xi$ is a vector field and $\eta$
is an 1-form, if
$$
\begin{array}{c} \varphi^{2}=-I+\eta\otimes\xi,\ \
\eta(\xi)=1,\ \ \forall X,Y\in C(TN).\end{array}
$$
\noindent A Riemannian metric $g$ on $N$ is said to be an
associated metric and then $(N,\varphi,\xi,\eta,g)$ is a
\textit{contact metric manifold} if
$$
g(\varphi X,\varphi Y)=g(X,Y)-\eta(X)\eta(Y),\ \ \ g(X,\varphi
Y)=d\eta(X,Y),\ \ \forall X,Y\in C(TN).
$$

\noindent A contact metric structure $(\varphi,\xi,\eta,g)$ is
called {\it normal} if
$$
N_{\varphi}+2d\eta\otimes\xi=0,
$$
where
$$
N_{\varphi}(X,Y)=[\varphi X,\varphi Y]-\varphi \lbrack \varphi
X,Y]-\varphi \lbrack X,\varphi Y]+\varphi ^{2}[X,Y],\ \ \forall
X,Y\in C(TN),
$$
is the Nijenhuis tensor field of $\varphi$.

\noindent A contact metric manifold $(N,\varphi,\xi,\eta,g)$ is a
\textit{Sasakian manifold} if it is normal or, equivalently, if
$$
(\nabla_{X}\varphi)(Y)=g(X,Y)\xi-\eta(Y)X,\ \ \forall X,Y\in
C(TN).
$$

\noindent The {\it contact distribution} of a Sasakian manifold
$(N,\varphi,\xi,\eta,g)$ is defined by $\{X\in TN:\eta(X)=0\}$,
and an integral curve of the contact distribution is called {\it
Legendre curve}.

\noindent Let $(N,\varphi,\xi,\eta,g)$ be a Sasakian manifold. The
sectional curvature of a 2-plane generated by $X$ and $\varphi X$,
where $X$ is an unit vector orthogonal to $\xi$, is called
\textit{$\varphi$-sectional curvature} determined by $X$. A
Sasakian manifold with constant $\varphi$-sectional curvature $c$
is called a \textit{Sasakian space form} and it is denoted by
$N(c)$.

\noindent The curvature tensor field of a Sasakian space form
$N(c)$ is given by
\begin{equation}\label{eccurv}
\begin{array}{ll}
R(X,Y)Z=&\frac{c+3}{4}\{g(Z,Y)X-g(Z,X)Y\}+\frac{c-1}{4}\{\eta(Z)\eta(X)Y-\\ \\
&-\eta(Z)\eta(Y)X+g(Z,X)\eta(Y)\xi-g(Z,Y)\eta(X)\xi+\\
\\&+g(Z,\varphi Y)\varphi X-g(Z,\varphi X)\varphi
Y+2g(X,\varphi Y)\varphi Z\}.
\end{array}
\end{equation}

\section{Biharmonic non-Legendre curves in Sasakian space forms}
\setcounter{equation}{0}

\begin{definition}
Let $(N^{m},g)$ be a Riemannian manifold and $\gamma:I\to N$ a
curve parametrized by arc length, that is $\vert\gamma'\vert=1$.
Then $\gamma$ is called a \textit{Frenet curve of osculating order
r}, $1\leq r\leq m$, if there exists orthonormal vector fields
$E_{1},E_{2},...,E_{r}$ along $\gamma$ such that $E_{1}=\gamma'=T$
and
$$
\nabla_{T}E_{1}=\kappa_{1}E_{2},\ \
\nabla_{T}E_{2}=-\kappa_{1}E_{1}+\kappa_{2}E_{3},...,
\nabla_{T}E_{r}=-\kappa_{r-1}E_{r-1},
$$
where $\kappa_{1},...,\kappa_{r-1}$ are positive functions on $I$.
\end{definition}

\noindent A Frenet curve of osculating order 1 is a geodesic; a
Frenet curve of osculating order 2 with $\kappa_{1}=\cst$ is
called a \textit{circle}; a Frenet curve of osculating order $r$,
$r\geq 3$, with $\kappa_{1},...,\kappa_{r-1}$ constants is called
a \textit{helix of order r} and helix of order $3$ is called,
simply, helix.

Let $(N^{2n+1},\varphi,\xi,\eta,g)$ be a Sasakian space form with
constant $\varphi$-sectional curvature $c$ and $\gamma:I\to N$ a
non-Legendre Frenet curve of osculating order $r$ with
$\eta(T)=f$, where $f$ is a function defined along $\gamma$ and
$f\neq 0$. Since
$$
\begin{array}{llc}
\nabla_{T}^{3}T=&(-3\kappa_{1}\kappa_{1}')E_{1}+(\kappa_{1}''-\kappa_{1}^{3}-\kappa_{1}\kappa_{2}^{2})E_{2}+(2\kappa_{1}'\kappa_{2}+\kappa_{1}\kappa_{2}')E_{3}\\
\\
&+\kappa_{1}\kappa_{2}\kappa_{3}E_{4}
\end{array}
$$
and
$$
\begin{array}{llc}
R(T,\nabla_{T}T)T=&\Big(-\frac{(c+3)\kappa_{1}}{4}+\frac{(c-1)\kappa_{1}}{4}f^{2}\Big)E_{2}-\frac{(c-1)}{4}ff'T+\frac{(c-1)}{4}f'\xi\\
\\&-\frac{3(c-1)\kappa_{1}}{4}g(E_{2},\varphi T)\varphi T,
\end{array}
$$
we get
\begin{equation}\label{eq:1}
\begin{array}{rl}
\tau_{2}(\gamma)=&\nabla_{T}^{3}T-R(T,\nabla_{T}T)T\\ \\
=&(-3\kappa_{1}\kappa_{1}'+\frac{(c-1)}{4}ff')E_{1}+\Big(\kappa_{1}''-\kappa_{1}^{3}-\kappa_{1}\kappa_{2}^{2}+\frac{(c+3)\kappa_{1}}{4}-\frac{(c-1)\kappa_{1}}{4}f^{2}\Big)E_{2}\\
\\&+(2\kappa_{1}'\kappa_{2}+\kappa_{1}\kappa_{2}')E_{3}+\kappa_{1}\kappa_{2}\kappa_{3}E_{4}-\frac{(c-1)}{4}f'\xi+\frac{3(c-1)\kappa_{1}}{4}g(E_{2},\varphi T)\varphi T.
\end{array}
\end{equation}

\noindent If $c=1$ then $\gamma$ is proper-biharmonic if and only
if
$$
\left\{\begin{array}{l} \kappa_{1}=\cst>0,\ \ \kappa_{2}=\cst\\ \\
\kappa_{1}^{2}+\kappa_{2}^{2}=1\\ \\
\kappa_{2}\kappa_{3}=0\end{array}\right.
$$
and we can state the following Theorem.

\begin{theorem} \label{teocase1}If $c=1$ then $\gamma$ is
proper-biharmonic if and only if either $\gamma$ is a circle with
$\kappa_{1}=1$, or $\gamma$ is a helix with
$\kappa_{1}^{2}+\kappa_{2}^{2}=1$.
\end{theorem}

Now, assume that $c\neq 1$. Then $\gamma$ is proper-bihrmonic if
and only if
\begin{equation}\label{eq:2}
\left\{\begin{array}{l} \kappa_{1}=\cst>0,\\ \\
\kappa_{1}^{2}+\kappa_{2}^{2}=\frac{c+3}{4}-\frac{c-1}{4}f^{2}-\frac{1}{\kappa^{2}_{1}}\frac{c-1}{4}(f')^{2}+\frac{3(c-1)}{4}(g(E_{2},\varphi
T))^{2}\\ \\
\kappa_{2}'-\frac{1}{\kappa_{1}}\frac{c-1}{4}f'\eta(E_{3})+\frac{3(c-1)}{4}g(E_{2},\varphi
T)g(E_{3},\varphi T)=0\\
\\\kappa_{2}\kappa_{3}-\frac{1}{\kappa_{1}}\frac{c-1}{4}f'\eta(E_{4})+\frac{3(c-1)}{4}g(E_{2},\varphi
T)g(E_{4},\varphi T)=0\end{array}\right.,
\end{equation}
since, from $\eta(T)=g(T,\xi)=f$ and the first Frenet equation, we
get $g(\nabla_{T}T,\xi)=g(\kappa_{1}E_{2},\xi)=f'$.

\noindent Obviously, the above equations are simpler when
$f=\eta(T)=\cos\beta_{1}$ is a constant, where
$\beta_{1}\in(0,\pi)\setminus\{\frac{\pi}{2}\}$ is the angle
between the tangent vector field $T$ and the characteristic vector
field $\xi$.

\noindent In the following, we will study only this special case.
\noindent We have

\begin{theorem} Let $c\neq 1$ and $\gamma$ a Frenet curve of osculating order $r$ such that
$\eta(T)=\cos\beta_{1}=\cst\notin\{-1,0,1\}$. Then $\gamma$ is
proper-biharmonic if and only if either

a) $\gamma$ is a circle with $\varphi T=\pm\sin\beta_{1}E_{2}$ and
$ \kappa^{2}_{1}=1+(c-1)\sin^{2}\beta_{1}>0$,

or

b) $\gamma$ is a helix with $\varphi T=\pm\sin\beta_{1}E_{2}$ and
$\kappa^{2}_{1}+\kappa_{2}^{2}=1+(c-1)\sin^{2}\beta_{1}>0$,

or

c) $\gamma$ is a Frenet curve of osculating order $r$, where
$r\geq 4$, with
$$
\varphi
T=\sin\beta_{1}\cos\beta_{2}E_{2}+\sin\beta_{1}\sin\beta_{2}E_{4}
$$
and
$$
\left\{\begin{array}{l}\kappa_{1}=\cst>0,\ \ \kappa_{2}=\cst\\ \\
\kappa_{1}^{2}+\kappa_{2}^{2}=\frac{c+3}{4}-\frac{c-1}{4}\cos^{2}\beta_{1}+\frac{3(c-1)}{4}\sin^{2}\beta_{1}\cos^{2}\beta_{2}\\
\\\kappa_{2}\kappa_{3}=-\frac{3(c-1)}{8}\sin^{2}\beta_{1}\sin(2\beta_{2})
\end{array}\right.,
$$
where $\beta_{2}\in (0,2\pi)$ is a constant such that
$c+3-(c-1)\cos^{2}\beta_{1}+3(c-1)sin^{2}\beta_{1}\cos^{2}\beta_{2}$
$>0$, $3(c-1)\sin(2\beta_{2})<0$ and $n\geq 2$.
\end{theorem}

\begin{proof} First, we see that
$\eta(E_{2})=g(E_{2},\xi)=\frac{1}{\kappa_{1}}f'=0$. Next, assume
that $g(E_{2},\varphi T)$ $=\alpha$, where $\alpha$ is a function
defined along $\gamma$. Then, using the second Frenet equation,
one obtains
$$
\alpha'=g(\nabla_{T}E_{2},\varphi T)+g(E_{2},\nabla_{T}\varphi
T)=\kappa_{2}g(E_{3},\varphi T)+g(E_{2},\kappa_{1}\varphi
E_{2}+\xi-fT),
$$
and, since the second term in the right side vanishes, it follows
$\kappa_{2}g(E_{3},\varphi T)=\alpha'$.

\noindent Now, if $\gamma$ is proper-biharmonic, replacing into
the third equation of (\ref{eq:2}) we obtain
$$
\kappa_{2}\kappa_{2}'+\frac{3(c-1)}{4}\alpha\alpha'=0
$$
and then
$$
\kappa^{2}_{2}+\frac{3(c-1)}{4}\alpha^{2}+\omega_{0}=0,
$$
where $\omega_{0}$ is a constant. The second equation of
(\ref{eq:2}) becomes
$$
\kappa_{1}^{2}+\kappa_{2}^{2}=\frac{c+3}{4}-\frac{c-1}{4}f^{2}-\kappa_{2}^{2}-\omega_{0}.
$$
\noindent Hence $\kappa_{2}=\cst$ and $\alpha=\cst$. If
$\kappa_{2}=0$ then, from the biharmonic equation
$\tau_{2}(\gamma)=0$, we get $E_{2}\parallel\varphi T$ and, since
$g(\varphi T,\varphi T)=1-f^{2}=\sin^{2}\beta_{1}$ it follows
$\varphi T=\pm\sin\beta_{1}E_{2}$. Hence $\gamma$ is a circle with
$$
\kappa^{2}_{1}=\frac{c+3}{4}-\frac{c-1}{4}\cos^{2}\beta_{1}+\frac{3(c-1)}{4}\sin^{2}\beta_{1}=1+(c-1)\sin^{2}\beta_{1}.
$$

\noindent Assume that $\kappa_{2}\neq 0$. Then $g(E_{3},\varphi
T)=0$ and, if $\kappa_{3}=0$ then $\gamma$ is a helix with
$$
\kappa^{2}_{1}+\kappa_{2}^{2}=\frac{c+3}{4}-\frac{c-1}{4}\cos^{2}\beta_{1}+\frac{3(c-1)}{4}\sin^{2}\beta_{1}=1+(c-1)\sin^{2}\beta_{1},
$$
since, using again the biharmonic equation, one obtains
$E_{2}\parallel \varphi T$ in this case too.

\noindent Next, let $\gamma$ be a proper-biharmonic Frenet curve
of osculating order $r$ with $r\geq 4$. Then $g(E_{3},\varphi
T)=0$ and, from the biharmonic equation we have $\varphi T\in
\Span\{E_{2},E_{4}\}$. Since
$$
g(\varphi T,\varphi T)=1-f^{2}=\sin^{2}\beta_{1}
$$
it follows
$$
\varphi
T=\sin\beta_{1}\cos\beta_{2}E_{2}+\sin\beta_{1}\sin\beta_{2}E_{4},
$$
where
$$
g(E_{2},\varphi T)=\alpha=\sin\beta_{1}\cos\beta_{2}\ \ \
\text{and}\ \ \ g(E_{4},\varphi T)=\sin\beta_{1}\sin\beta_{2}
$$
with $\beta_{2}=\cst\in(0,2\pi)$.

\noindent Finally, if the dimension of $N$ is equal to $3$ we can
consider an orthogonal system of vectors $\{E=T-f\xi,\varphi
T,\xi\}$ along $\gamma$ and, since $f$ is a constant, it follows
easily that $\nabla_{T}T\parallel\varphi T$. Hence
$E_{2}\parallel\varphi T$ in this case.

\end{proof}

A special role in the biharmonic equation $\tau_{2}(\gamma)=0$ is
played by $g(E_{2},\varphi T)$. In the following, we consider
$\gamma$ to be a Frenet curve of osculating order $r$, with
$\eta(T)=f(s)=\cos\beta(s)$ not necessarily constant, such that
$E_{2}\perp\varphi T$ or $E_{2}\parallel \varphi T$.

\noindent\textbf{Case I: $\mathbf{c\neq 1,\ E_{2}\perp\varphi
T}$.}

\noindent In this case $\gamma$ is proper-biharmonic if and only
if
\begin{equation}\label{eq:3}
\left\{\begin{array}{l} \kappa_{1}=\cst>0,\\ \\
\kappa_{1}^{2}+\kappa_{2}^{2}=\frac{c+3}{4}-\frac{c-1}{4}f^{2}-\frac{1}{\kappa^{2}_{1}}\frac{c-1}{4}(f')^{2}\\ \\
\kappa_{2}'-\frac{1}{\kappa_{1}}\frac{c-1}{4}f'\eta(E_{3})=0\\
\\\kappa_{2}\kappa_{3}-\frac{1}{\kappa_{1}}\frac{c-1}{4}f'\eta(E_{4})=0\end{array}\right.,
\end{equation}

\noindent From $g(E_{2},\xi)=\frac{1}{\kappa_{1}}f'$ one obtains
$g(\nabla_{T}E_{2},\xi)-g(E_{2},\varphi
T)=\frac{1}{\kappa_{1}}f''$ and then $\kappa_{2}\eta(E_{3})$
$=\frac{1}{\kappa_{1}}f''+\kappa_{1}f$. Replacing into the third
equation of \eqref{eq:3} one obtains
$$
\kappa_{2}\kappa_{2}'-\frac{1}{\kappa^{2}_{1}}\frac{c-1}{4}(f'f''+\kappa^{2}_{1}ff')=0,
$$
and then
$$
\kappa_{2}^{2}-\frac{1}{\kappa^{2}_{1}}\frac{c-1}{4}((f')^{2}+\kappa^{2}_{1}f^{2})+\omega_{1}=0,
$$
where $\omega_{1}$ is a constant. Now, from the second equation of
\eqref{eq:3} we have
$$
\kappa_{1}^{2}+\kappa_{2}^{2}=\frac{c+3}{4}-\kappa^{2}_{2}-\omega_{1}.
$$
\noindent Hence $\kappa_{2}=\cst$ and $(f''+\kappa^{2}_{1}f)f'=0$.

Now, using the Frenet equations, from $g(E_{2},\varphi T)=0$ one
obtains $\kappa_{2}g(E_{3},\varphi T)=-\frac{1}{\kappa_{1}}f'$ and
then, from
$\kappa_{2}g(E_{3},\xi)=\frac{1}{\kappa_{1}}f''+\kappa_{1}f$, we
get
$$
\kappa_{2}\kappa_{3}g(E_{4},\xi)=\frac{1}{\kappa_{1}}(f'''+(\kappa^{2}_{1}+\kappa^{2}_{2})f').
$$

\noindent Since $\tau_{2}(\gamma)=0$ implies
$\eta(\tau_{2}(\gamma))=0$ one obtains, after a straightforward
computation that $f'f'''=0$. Using this result and differentiating
$(f''+\kappa^{2}_{1}f)f'=0$ along $\gamma$ we have
$$
\kappa^{2}_{1}(f')^{2}+(f''+\kappa^{2}_{1}f)f''=0.
$$
We just obtained that $f=\cst$.

\noindent We can state

\begin{theorem}\label{tperp} Assume that $c\neq 1$, $n\geq 2$ and $\nabla_{T}T\perp\varphi
T$. Then $\gamma$ is proper-biharmonic if and only if either

a) $\gamma$ is a circle with $\eta(T)=\cos\beta_{0}$ and
$\kappa_{1}^{2}=\frac{c+3}{4}-\frac{c-1}{4}\cos^{2}\beta_{0}$,

or

b) $\gamma$ is a helix with $\eta(T)=\cos\beta_{0}$ and
$\kappa_{1}^{2}+\kappa_{2}^{2}=\frac{c+3}{4}-\frac{c-1}{4}\cos^{2}\beta_{0}$,

\noindent where
$\beta_{0}\in(0,2\pi)\setminus\{\frac{\pi}{2},\pi,\frac{3\pi}{2}\}$
is a constant such that
$\frac{c+3}{4}-\frac{c-1}{4}\cos^{2}\beta_{0}>0$.
\end{theorem}

\begin{remark} We note that the biharmonic equation
$\tau_{2}(\gamma)=0$ for curves $\gamma$ with
$\nabla_{T}T\perp\varphi T$ is equivalent to
$$
\Delta H=\frac{1}{4}(c+3-(c-1)\cos^{2}\beta_{0})H,
$$
i.e. $H$ is an eigenvector of $\Delta$, where $H=\nabla_{T}T$ is
the mean curvature vector field of $\gamma$.
\end{remark}

\noindent\textbf{Case II: $\mathbf{c\neq 1,\ E_{2}\parallel\varphi
T}$.}

\noindent In this case $g(E_{2},\xi)=\frac{1}{\kappa_{1}}f'=0$ and
then $f=\cos\beta_{0}=\cst$. Since $g(\varphi T,\varphi
T)=1-(g(T,T))^{2}=\sin^{2}\beta_{0}$ we have $\varphi
T=\pm\sin\beta_{0}E_{2}$.

\noindent We obtain
\begin{proposition}\label{p1} Assume that $c\neq 1$ and $\nabla_{T}T\parallel\varphi
T$. Then $\gamma$ is proper-biharmonic if and only if either

a) $\gamma$ is a circle with $\eta(T)=\cos\beta_{0}$ and
$\kappa_{1}^{2}=c-(c-1)\cos^{2}\beta_{0}$,

or

b) $\gamma$ is a helix with $\eta(T)=\cos\beta_{0}$ and
$\kappa_{1}^{2}+\kappa_{2}^{2}=c-(c-1)\cos^{2}\beta_{0}$,

\noindent where
$\beta_{0}\in(0,2\pi)\setminus\{\frac{\pi}{2},\pi,\frac{3\pi}{2}\}$
is a constant such that $c-(c-1)\cos^{2}\beta_{0}>0$.
\end{proposition}

\noindent Next, let $\gamma$ be a proper-biharmonic non-Legendre
curve with $\nabla_{T}T\parallel\varphi T$. As $\varphi
T=\pm\sin\beta_{0}E_{2}$ one obtains after a straightforward
computation that
$$
\nabla_{T}E_{2}=-\frac{1}{\sin\beta_{0}}\Big(\frac{\kappa_{1}}{\sin\beta_{0}}\pm\cos\beta_{0}\Big)T+
\frac{1}{\sin\beta_{0}}\Big(\frac{\kappa_{1}\cos\beta_{0}}{\sin\beta_{0}}\pm
1\Big)\xi.
$$
\noindent Using the second Frenet equation we have
$$
\kappa_{2}^{2}=\frac{(\kappa_{1}\cos\beta_{0}\pm\sin\beta_{0})^{2}}{\sin^{2}\beta_{0}}.
$$
Thus $\gamma$ is a circle if and only if
$\kappa_{1}=\mp\tan\beta_{0}>0$. From Proposition \ref{p1} we
easily get that $\gamma$ is a proper-biharmonic circle if and only
if
$$
\kappa_{1}^{2}=\frac{c-1+\sqrt{c^{2}-2c+5}}{2}\ \ \text{and}\ \
\cos^{2}\beta_{0}=\frac{c+1-\sqrt{c^{2}-2c+5}}{2(c-1)}.
$$

\noindent If $\kappa_{2}\neq 0$, from the expression of
$\kappa_{2}$ and the third Frenet equation it follows that
$\kappa_{3}=0$. Hence $\gamma$ is a helix. Now, $\gamma$ is
proper-biharmonic if and only if $\kappa_{1}$ satisfies
$$
\kappa_{1}^{2}\pm\cos(2\beta_{0})\kappa_{1}+(1-c)\sin^{4}\beta_{0}=0
$$
and
$\beta_{0}\in(0,2\pi)\setminus\{\frac{\pi}{2},\pi,\frac{3\pi}{2}\}$
if $c>1$ or
$\beta_{0}\in(0,2\pi)\setminus\{\frac{\pi}{2},\pi,\frac{3\pi}{2}\}$
such that
$\cos\beta_{0}\in\Big(-\sqrt{\frac{c-1}{c-2}},\sqrt{\frac{c-1}{c-2}}\Big)$
if $c<1$.

\noindent We conclude with the following

\begin{theorem}\label{tpar} If $c\neq 1$ and $\nabla_{T}T\parallel\varphi
T$, then $\gamma$ is a Frenet curve of osculating order $r\leq 3$
and it is proper-biharmonic if and only if either

a) $\gamma$ is a circle with
$\eta(T)=\pm\sqrt{\frac{c+1-\sqrt{c^{2}-2c+5}}{2(c-1)}}$ and
$\kappa_{1}^{2}=\frac{c-1+\sqrt{c^{2}-2c+5}}{2}$,

or

b) $\gamma$ is a helix with $\eta(T)=\cos\beta_{0}$ and
$\kappa_{1}$ satisfies
$$
\kappa_{1}^{2}\pm\cos(2\beta_{0})\kappa_{1}+(1-c)\sin^{4}\beta_{0}=0,
$$
where
$\beta_{0}=\cst\in(0,2\pi)\setminus\{\frac{\pi}{2},\pi,\frac{3\pi}{2}\}$
if $c>1$ or
$\beta_{0}=\cst\in(0,2\pi)\setminus\{\frac{\pi}{2},\pi,\frac{3\pi}{2}\}$
such that
$\cos\beta_{0}\in\Big(-\sqrt{\frac{c-1}{c-2}},\sqrt{\frac{c-1}{c-2}}\Big)$
if $c<1$. In the last case
$\kappa^{2}_{2}=(\kappa_{1}\cot\beta_{0}\pm 1)^{2}$.
\end{theorem}

\begin{remark} A curve $\gamma$ with
$\nabla_{T}T\parallel\varphi T$ is proper-biharmonic if and only
if
$$
\Delta H=(c-(c-1)\cos^{2}\beta_{0})H,
$$
where $H$ is the mean curvature vector field of $\gamma$.
\end{remark}

\section{Biharmonic curves in $\mathbb{R}^{2n+1}(-3)$}
\setcounter{equation}{0}

While proper-biharmonic Legendre curves exist only in a Sasakian
space form $N^{2n+1}(c)$ with constant $\varphi$-sectional
curvature $c$ bigger than $1$ if $n=1$, or $-3$ if $n>1$ (see
\cite{Ino}, \cite{FetcuOniciuc2}), proper-biharmonic non-Legendre
curves can be found in Sasakian space forms with any
$\varphi$-sectional curvature.

\noindent We mention that, in the case when $c=-3$, T. Sasahara
studied in \cite{Sas1} the submanifolds in the Sasakian space form
$\mathbb{R}^{2n+1}(-3)$ whose $\varphi$-mean curvature vectors are
eigenvectors of the Laplacian and in \cite{Sas2} the Legendre
surfaces in $\mathbb{R}^{5}(-3)$ for which mean curvature vectors
field are eigenvectors of the Laplacian.

\noindent In this section we obtain the explicit equations for
proper-biharmonic circles with $E_{2}\perp\varphi T$ and for all
proper-biharmonic curves with $E_{2}\parallel\varphi T$ in
$\mathbb{R}^{2n+1}(-3)$.

First, let us recall briefly some notions and results about the
structure of the Sasakian space form $\mathbb{R}^{2n+1}(-3)$ as
they are presented in \cite{Blair}.

\noindent Consider on $\mathbb{R}^{2n+1}(-3)$, with elements of
the form $(x^{1},...,x^{n},y^{1},...,y^{n},z)$, its standard
contact structure defined by the 1-form
$\eta=\frac{1}{2}(dz-\sum_{i=1}^{n}y^{i}dx^{i})$, the
characteristic vector field $\xi=2\frac{\partial}{\partial z}$ and
the tensor field $\varphi$ given by the matrix
$$
\left(\begin{array}{ccc}0&\delta_{ij}&0\\ -\delta_{ij}&0&0\\
0&y^{j}&0\end{array}\right).
$$
Then
$g=\eta\otimes\eta+\frac{1}{4}\sum_{i=1}^{n}((dx^{i})^{2}+(dy^{i})^{2})$
is an associated Riemannian metric and
$(\mathbb{R}^{2n+1},\varphi,\xi,\eta,g)$ is a Sasakian space form
with constant $\varphi$-sectional curvature equal to $-3$, denoted
$\mathbb{R}^{2n+1}(-3)$.

\noindent The vector fields $X_{i}=2\frac{\partial}{\partial
y^{i}}$, $X_{n+i}=\varphi X_{i}=2(\frac{\partial}{\partial
x^{i}}+y^{i}\frac{\partial}{\partial z})$, $i=1,...,n$, and
$\xi=2\frac{\partial}{\partial z}$ form an orthonormal basis in
$\mathbb{R}^{2n+1}(-3)$ and after straightforward computations one
obtains
$$
[X_{i},X_{j}]=[X_{n+i},X_{n+j}]=[X_{i},\xi]=[X_{n+i},\xi]=0,\ \ \
[X_{i},X_{n+j}]=2\delta_{ij}\xi
$$
and
$$
\nabla_{X_{i}}X_{j}=\nabla_{X_{n+i}}X_{n+j}=0,\ \
\nabla_{X_{i}}X_{n+j}=\delta_{ij}\xi,\ \
\nabla_{X_{n+i}}X_{j}=-\delta_{ij}\xi,
$$
$$
\nabla_{X_{i}}\xi=\nabla_{\xi}X_{i}=-X_{n+i},\ \
\nabla_{X_{n+i}}\xi=\nabla_{\xi}X_{n+i}=X_{i}
$$
for any $i,j=1,...,n$.

Now, let $\gamma:I\to \mathbb{R}^{2n+1}(-3)$ be a Frenet curve of
osculating order $r>1$, parametrized by arc length, with the
tangent vector field $T=\gamma'$ given by
\begin{equation}\label{eq:4}
T=\sum_{i=1}^{n}(T_{i}X_{i}+T_{n+i}X_{n+i})+\cos\beta_{0}\xi,
\end{equation}
where $\cos\beta_{0}$ is a constant. Using the above formulas for
the Levi-Civita connection we have
\begin{equation}\label{eq:5}
\nabla_{T}T=\sum_{i=1}^{n}((T'_{i}+2\cos\beta_{0}T_{n+i})X_{i}+(T'_{n+i}-2\cos\beta_{0}T_{i})X_{n+i})
\end{equation}

From Theorems \ref{tperp} and \ref{tpar}, using the same
techniques as in \cite{Cad1}, \cite{Cad2} and \cite{CIL}, we get

\begin{theorem} The parametric equations of proper-biharmonic
circles parametrized by arc length in $\mathbb{R}^{2n+1}(-3)$,
$n\geq 2$, with $\nabla_{T}T\perp\varphi T$,  are
\begin{equation}\label{eq:6}
\left\{\begin{array}{lll}
x^{i}(s)&=&\pm\frac{1}{\kappa_{1}}(2\sin(\kappa_{1}s)c_{1}^{i}\mp
2\cos(\kappa_{1}s)c_{2}^{i}-\cos(2\kappa_{1}s)d_{1}^{i}\\ \\&&-\sin(2\kappa_{1}s)d_{2}^{i})+a^{i}\\
\\ y^{i}(s)&=&\frac{1}{\kappa_{1}}(2\cos(\kappa_{1}s)c_{1}^{i}\pm
2\sin(\kappa_{1}s)c_{2}^{i}+\sin(2\kappa_{1}s)d_{1}^{i}\\ \\&&-\cos(2\kappa_{1}s)d_{2}^{i})+b^{i}\\
\\
z(s)&=&\pm\frac{2}{\kappa_{1}}(1+\sum_{i=1}^{n}((c_{1}^{i})^{2}+(c_{2}^{i})^{2}))s\\
\\&&+\frac{1}{2\kappa_{1}^{2}}\sum_{i=1}^{n}(\pm\cos(4\kappa_{1}s)d_{1}^{i}d_{2}^{i}-2\cos(2\kappa_{1}s)c_{1}^{i}c_{2}^{i}\\ \\
&&+4\cos(3\kappa_{1}s)c_{2}^{i}d_{2}^{i}-4\sin(3\kappa_{1}s)c_{1}^{i}d_{2}^{i})\\
\\&&\mp\frac{1}{\kappa_{1}}\sum_{i=1}^{n}b^{i}(-2\sin(\kappa_{1}s)c_{1}^{i}\pm 2\cos(\kappa_{1}s)c_{2}^{i}\\ \\
&&+\cos(2\kappa_{1}s)d_{1}^{i}+\sin(2\kappa_{1}s)d_{2}^{i})+e
\end{array}\right.,
\end{equation}
where $\kappa_{1}^{2}=\cos^{2}\beta_{0}$,
$\beta_{0}\in(0,2\pi)\setminus\{\frac{\pi}{2},\pi,\frac{3\pi}{2}\}$
is a constant, and $a^{i}$, $b^{i}$, $c_{1}^{i}$, $c_{2}^{i}$,
$d_{1}^{i}$, $d_{2}^{i}$ and $e$ are constants such that the
$n$-dimensional constant vectors $c_{j}=(c_{j}^{1},...,c_{j}^{n})$
and $d_{j}=(d_{j}^{1},...,d_{j}^{n})$, $j=1,2$, satisfy
$$
\left\{\begin{array}{lc}\vert c_{1}\vert^{2}+\vert c_{2}\vert^{2}+\vert d_{1}\vert^{2}+\vert d_{2}\vert^{2}=\sin^{2}\beta_{0}\\
\\
\langle c_{1},d_{1}\rangle\pm\langle c_{2},d_{2}\rangle=0,\
\langle c_{1},d_{2}\rangle\mp\langle
c_{2},d_{1}\rangle=0\end{array}\right..
$$
\end{theorem}

\begin{proof} Let $\gamma:I\to \mathbb{R}^{2n+1}(-3)$ be a circle parametrized by arc length, with the
tangent vector field $T=\gamma'$ given by \eqref{eq:4} and
$\nabla_{T}T\perp\varphi T$. From the equation \eqref{eq:5} one
obtains
$$
E_{2}=\frac{1}{\kappa_{1}}\sum_{i=1}^{n}((T'_{i}+2\cos\beta_{0}T_{n+i})X_{i}+(T'_{n+i}-2\cos\beta_{0}T_{i})X_{n+i})
$$
and, using $g(E_{2},\varphi T)=0$, a direct computation shows that
$$
\begin{array}{lll}
\nabla_{T}E_{2}&=&\frac{1}{\kappa_{1}}(\sum_{i=1}^{n}((T'_{i}+2\cos\beta_{0}T_{n+i})'+(T'_{n+i}-2\cos\beta_{0}T_{i})\cos\beta_{0})X_{i}\\
\\&&+((T'_{n+i}-2\cos\beta_{0}T_{i})'-(T'_{i}+2\cos\beta_{0}T_{n+i})\cos\beta_{0})X_{n+i})
\end{array}
$$
and, since $\gamma$ is a circle, it follows
\begin{equation}\label{eq:7}
\left\{\begin{array}{c} A'_{i}+B_{i}\cos\beta_{0}=0\\ \\
B'_{i}-A_{i}\cos\beta_{0}=0
\end{array}\right.,
\end{equation}
where $A_{i}=\frac{1}{\kappa_{1}}(T'_{i}+2\cos\beta_{0}T_{n+i})$
and $B_{i}=\frac{1}{\kappa_{1}}(T'_{n+i}-2\cos\beta_{0}T_{i})$.

\noindent Solving \eqref{eq:7} and imposing for $\gamma$ to be
proper-biharmonic, according to Theorem \ref{tperp} that is
$\kappa_{1}=\pm\cos\beta_{0}>0$, we get the following equations
$$
\left\{\begin{array}{c} T'_{i}\pm
2\kappa_{1}T_{n+i}=\kappa_{1}\cos(\kappa_{1}s)c_{1}^{i}\pm\kappa_{1}\sin(\kappa_{1}s)c_{2}^{i}\\
\\ T'_{n+i}\mp
2\kappa_{1}T_{i}=\pm\kappa_{1}\sin(\kappa_{1}s)c_{1}^{i}-\kappa_{1}\cos(\kappa_{1}s)c_{2}^{i}
\end{array}\right.,
$$
which general solutions are
$$
\left\{\begin{array}{c}
T_{i}=-\sin(\kappa_{1}s)c_{1}^{i}\pm\cos(\kappa_{1}s)c_{2}^{i}+\cos(2\kappa_{1}s)d_{1}^{i}+\sin(2\kappa_{1}s)d_{2}^{i}\\
\\ T_{n+i}=\pm\cos(\kappa_{1}s)c_{1}^{i}+\sin(\kappa_{1}s)c_{2}^{i}\pm\sin(2\kappa_{1}s)d_{1}^{i}\mp\cos(2\kappa_{1}s)d_{2}^{i}
\end{array}\right.,
$$
where $c_{1}^{i}$, $c_{2}^{i}$, $d_{1}^{i}$ and $d_{2}^{i}$ are
constants, such that
$$
\left\{\begin{array}{lc}\sum_{i=1}^{n}((c_{1}^{i})^{2}+(c_{2}^{i})^{2}+(d_{1}^{i})^{2}+(d_{2}^{i})^{2})=\sin^{2}\beta_{0}\\
\\
\sum_{i=1}^{n}((c_{1}^{i})(d_{1}^{i})\pm(c_{2}^{i})(d_{2}^{i}))=0,\
\ \
\sum_{i=1}^{n}((c_{1}^{i})(d_{2}^{i})\mp(c_{2}^{i})(d_{1}^{i}))=0\end{array}\right.,
$$
since $g(T,T)=1$.

\noindent Finally, replacing into expression of $\gamma'$ and
integrating we get \eqref{eq:6}.
\end{proof}

\begin{remark} In order to find explicit examples of proper-biharmonic
curves with $\nabla_{T}T\perp\varphi T$ in $\mathbb{R}^{2n+1}(-3)$
we will stick at proper-biharmonic circles since the computations
in the case of helices are rather complicated.
\end{remark}

\begin{theorem} Proper-biharmonic
curves in $\mathbb{R}^{2n+1}(-3)$, with
$\nabla_{T}T\parallel\varphi T$, are either

a) Proper-biharmonic circles given by
\begin{equation}\label{eq:8}
\left\{\begin{array}{ll}
x^{i}(s)=&(\sqrt{5}+1)\Big(\cos\Big(\frac{\sqrt{5}-1}{2}s\Big)c_{1}^{i}+\sin\Big(\frac{\sqrt{5}-1}{2}s\Big)c_{2}^{i}\Big)+a^{i}\\
\\ y^{i}(s)=&(\sqrt{5}+1)\Big(\sin\Big(\frac{\sqrt{5}-1}{2}s\Big)c_{1}^{i}-\cos\Big(\frac{\sqrt{5}-1}{2}s\Big)c_{2}^{i}\Big)+b^{i}\\
\\ z(s)=&\frac{1-\sqrt{5}\pm
2\sqrt{1+\sqrt{5}}}{2}s+\frac{3+\sqrt{5}}{2}\sum_{i=1}^{n}(((c_{1}^{i})^{2}-(c_{2}^{i})^{2})\sin((\sqrt{5}-1)s)\\
\\&-2\cos((\sqrt{5}-1)s)c_{1}^{i}c_{2}^{i})+(1+\sqrt{5})\sum_{i=1}^{n}b_{i}\Big(\sin\Big(\frac{\sqrt{5}-1}{2}s\Big)c_{2}^{i}\\ \\&+
\cos\Big(\frac{\sqrt{5}-1}{2}s\Big)c_{1}^{i}\Big)+d
\end{array}\right.,
\end{equation}
where $a^{i}$, $b^{i}$, $c_{1}^{i}$, $c_{2}^{i}$ and $d$ are
constants such that the $n$-dimensional constant vectors
$c_{j}=(c_{j}^{1},...,c_{j}^{n})$, $j=1,2$, satisfy
$$
\vert c_{1}\vert^{2}+\vert c_{2}\vert^{2}=\frac{3-\sqrt{5}}{4}.
$$

or

b) Proper-biharmonic helices given by
\begin{equation}\label{eq:9}
\left\{\begin{array}{ll}
x^{i}(s)=&-\frac{2\kappa_{1}}{\kappa_{1}\pm\sin(2\beta_{0})}\Big(\cos\Big(\frac{\kappa_{1}\pm\sin(2\beta_{0})}{\kappa_{1}}s\Big)c_{1}^{i}
+\sin\Big(\frac{\kappa_{1}\pm\sin(2\beta_{0})}{\kappa_{1}}s\Big)c_{2}^{i}\Big)+a^{i}\\
\\ y^{i}(s)=&\frac{2\kappa_{1}}{\kappa_{1}\pm\sin(2\beta_{0})}\Big(\sin\Big(\frac{\kappa_{1}\pm\sin(2\beta_{0})}{\kappa_{1}}s\Big)c_{1}^{i}
-\cos\Big(\frac{\kappa_{1}\pm\sin(2\beta_{0})}{\kappa_{1}}s\Big)c_{2}^{i}\Big)+b^{i}\\
\\z(s)=&2\Big(\cos\beta_{0}+\frac{\kappa_{1}\sin^{2}\beta_{0}}{\kappa_{1}\pm\sin(2\beta_{0})}\Big)s+
\frac{\kappa^{2}_{1}}{(\kappa_{1}\pm\sin(2\beta_{0}))^{2}}\\
\\&\cdot\Big(\sin\Big(\frac{2(\kappa_{1}\pm\sin(2\beta_{0}))}{\kappa_{1}}s\Big)\sum_{i=1}^{n}((c_{1}^{i})^{2}-(c_{2}^{i})^{2})\\
\\&+\cos\Big(\frac{2(\kappa_{1}\pm\sin(2\beta_{0}))}{\kappa_{1}}s\Big)\sum_{i=1}^{n}(c_{1}^{i}c_{2}^{i})\Big)\\
\\&-\frac{2\kappa_{1}}{\kappa_{1}\pm\sin(2\beta_{0})}\sum_{i=1}^{n}b^{i}\Big(\cos\Big(\frac{\kappa_{1}\pm\sin(2\beta_{0})}{\kappa_{1}}s\Big)c_{1}^{i}
\\ \\&+\sin\Big(\frac{\kappa_{1}\pm\sin(2\beta_{0})}{\kappa_{1}}s\Big)c_{2}^{i}\Big)+d
\end{array}\right.,
\end{equation}
where
$\beta_{0}\in(0,2\pi)\setminus\{\frac{\pi}{2},\pi,\frac{3\pi}{2}\}$
is a constant such that
$\cos\beta_{0}\in\Big(-1,-\frac{2\sqrt{5}}{5}\Big)\cup\Big(\frac{2\sqrt{5}}{5},1\Big)$,
$\kappa_{1}$ is a positive solution of the equation
$$
\kappa_{1}^{2}\pm\sin(2\beta_{0})\kappa_{1}+4\sin^{4}\beta_{0}=0
$$
and $a^{i}$, $b^{i}$, $c_{1}^{i}$, $c_{2}^{i}$ and $d$ are
constants such that
$$
\vert c_{1}\vert^{2}+\vert c_{2}\vert^{2}=\sin^{2}\beta_{0}.
$$

\end{theorem}

\begin{proof} We will prove only the first statement because the
second one can be obtained in a similar way by the meaning of
Theorem \ref{tpar}.

\noindent Assume that $\gamma$ is a proper-biharmonic circle in
$\mathbb{R}^{2n+1}(-3)$ parametrized by arc length, such that
$\nabla_{T}T\parallel\varphi T$. Then, from \eqref{eq:5} and since
$\varphi T=\sum_{i=1}^{n}(-T_{n+i}X_{i}+T_{i}X_{n+i})$, $g(\varphi
T,\varphi T)=\sin^{2}\beta_{0}$, where $\eta(T)=\cos\beta_{0}$,
one obtains
$$
T_{i}'=\Big(\mp\frac{\sin(2\beta_{0})}{\kappa_{1}}-1\Big)T_{n+i},\
\ \ \
T_{n+i}'=\Big(\pm\frac{\sin(2\beta_{0})}{\kappa_{1}}+1\Big)T_{i}
$$
Now, since $\gamma$ is a proper-biharmonic circle we get, from
Theorem \ref{tpar}, $\kappa_{1}=\mp\tan\beta_{0}>0$ and
$\cos^{2}\beta_{0}=\frac{1+\sqrt{5}}{4}$ and hence the above
equations become
$$
T_{i}'=\frac{\sqrt{5}-1}{2}T_{n+i},\ \ \ \
T_{n+i}'=\frac{1-\sqrt{5}}{2}T_{i},
$$
with general solutions
$$
T_{i}=\cos\Big(\frac{\sqrt{5}-1}{2}\Big)c_{1}^{i}+\sin\Big(\frac{\sqrt{5}-1}{2}\Big)c_{2}^{i},
\ \
T_{n+i}=\cos\Big(\frac{\sqrt{5}-1}{2}\Big)c_{2}^{i}-\sin\Big(\frac{\sqrt{5}-1}{2}\Big)c_{1}^{i},
$$
where $c_{1}^{i}$ and $c_{2}^{i}$, $i=1,...,n$, are constants.

\noindent Replacing in the expression of $T=\gamma'$, integrating
and imposing $g(T,T)=1$ we obtain the conclusion.

\end{proof}

\end{document}